\documentclass[letterpaper,11pt]{amsart}

\usepackage[all]{xy}                      %
  
\CompileMatrices                            

\UseTips                                    

\input xypic
\usepackage[bookmarks=true]{hyperref}       
\usepackage{tikz-cd}

\usepackage{amssymb,latexsym,amsmath,amscd}
\usepackage{xspace}
\usepackage{color}
\usepackage{graphicx}
\usepackage{dsfont}

\reversemarginpar

\DeclareMathOperator{\Res}{Res}

\theoremstyle{plain}
\newtheorem{theorem}{Theorem}[section]
\newtheorem*{theorem*}{Theorem}
\newtheorem{proposition}[theorem]{Proposition}

\newtheorem{lemma}[theorem]{Lemma}
\newtheorem{conjecture}[theorem]{Conjecture}

\theoremstyle{definition}
\newtheorem{definition}[theorem]{Definition}

\newtheorem{notation}[theorem]{Notation}

\newtheorem{remark}[theorem]{Remark}
\newtheorem{example}[theorem]{Example}

\newcommand{\enm}[1]{\ensuremath{#1}}          %

\newcommand{\cal}[1]{\mathcal{#1}}

\newcommand{\CC}{\enm{\mathbb{C}}}

\newcommand{\NN}{\enm{\mathbb{N}}}

\newcommand{\ZZ}{\enm{\mathbb{Z}}}

\newcommand{\PP}{\enm{\mathbb{P}}}

\newcommand{\TT}{\enm{\mathbb{T}}}

\newcommand{\Ii}{\enm{\cal{I}}}

\newcommand{\Oo}{\enm{\cal{O}}}

\renewcommand{\phi}{\varphi}
\renewcommand{\theta}{\vartheta}
\renewcommand{\epsilon}{\varepsilon}

\DeclareMathOperator{\red}{red}

\renewcommand{\to}[1][]{\xrightarrow{\ #1\ }}

\newcommand{\old}[1]{}

\title{Minimal Terracini loci in the plane:\\ gaps and non-gaps}

\author[E. Ballico]{Edoardo Ballico}\address{Dipartimento Di Matematica,
  Universit\`a di Trento, Via Sommarive 14, 38123, Povo, Trento, Italia}
\email{edoardo.ballico@unitn.it}

\author[M.C. Brambilla]{Maria Chiara Brambilla}\address{
Universit\`a Politecnica delle Marche, via Brecce Bianche, I-60131 Ancona, Italia}
\email{m.c.brambilla@univpm.it}

\thanks{Partially supported by GNSAGA of INdAM}

\subjclass[2010]{Primary: 14C20; Secondary:14N07} 
\keywords{interpolation problems, minimal Terracini locus, Terracini locus, zero-dimensional schemes}
\begin{document}

\maketitle

{\it Dedicated to Enrique Arrondo on the occasion of his sixtieth birthday}


\begin{abstract} 
We study minimally Terracini finite sets of points in the projective plane and we prove that the sequence of the cardinalities of minimally Terracini sets can have any number of gaps for degree great enough.
\end{abstract}

\section{Introduction}
Terracini loci in projective varieties have been recently introduced in \cite{BC21, bbs} and 
 subsequentely studied e.g.\ in 
\cite{BC2, GS, BV, G, GSTTT}.

A collection of double points is called a {\it Terracini locus} when they are in special position, i.e.\ when they impose an unexpected number of conditions to a linear system of a given degree.
In the theory of higher secant varieties of projective varieties (see \cite{BCCGO} for a general reference), a set of points is Terracini if they are {\it dependent} in the sense that the tangent spaces at these points span a linear space of dimension less than the expected one.

The study of Terracini loci in the cases of Veronese and Segre varieties 
is relevant also for applications to the theory of tensors (rank, decomposition and identifiability of tensors), see e.g.\
\cite{CG}.

In this note, we focus on the case of $\PP^n$ and, following \cite[Definition 2.8]{bbs}, we impose the additional condition that the sets of points span the whole space. 
Indeed a set of points spanning a smaller linear subspace is Terracini in such subspace, by \cite[Proposition 3.2]{Terracini1}. 

More precisely we give the following definition. 
For any $x\in\NN$, let $S(\PP^n,x):=\{\text{all the sets of $x$ points of $\PP^n$}\}$.
\begin{definition}\label{def1}
We call {\it Terracini locus} and we denote by $\TT(n,d;x)$ the set of all $S\in S(\PP^n,x)$ such that 
\begin{itemize}
\item
$h^0(\Ii_{2S}(d)) >0$ and $h^1(\Ii_{2S}(d)) >0$.
\item  $\langle S\rangle =\PP^n$.
\end{itemize}
We say that $S$  is 
{\it Terracini with respect to} $\Oo_{\PP^n}(d)$ if $S\in\TT(n,d;x)$.
\end{definition}

Given a Terracini set of cardinality $x$  it is easy to construct Terracini sets of any cardinality $y\ge x$.
Hence, in an attempt to classify the possible cardinalities of Terracini loci, we need a more restrictive definition.
For this reason, in \cite{Terracini1}
we have introduced the notion of {\it minimal Terracini loci}.
\begin{definition}\label{def-min1}
We call {\it minimal Terracini locus} and we denote by $\TT(n,d;x)'$ the set of all $S\in \TT(n,d;x)$, such that
\begin{itemize}
\item $h^1(\Ii_{2A}(d)) =0$ for all $A\subsetneq S$.
\end{itemize}
We say that $S$ is  {\it minimally Terracini with respect to $\Oo_{\PP^n}(d)$} if $S\in\TT(n,d;x)'$.
\end{definition}

In this note we consider minimally Terracini sets in the projective plane $\PP^2$ and
our main result is the following theorem.
\begin{theorem}\label{due2} 
Fix an integer $c\ge2$. Then for all $d\ge 14c+2$ there are integers $x_i$, $1\le i\le c$, and $y_i$, $1\le i\le c-1$, such
that \begin{itemize}
\item
$x_1<y_1<x_2<y_2<\cdots <x_{c-1} <y_{c-1} < x_c,$ 
\item $\TT(2,d;x_i)'\ne \emptyset$ for all $i=1,\dots ,c$, and
\item
 $\TT(2,d;y_i)'=\emptyset$ for all $i=1,\dots ,c-1$.
\end{itemize}
\end{theorem}

{We expect the same phenomenon to occur also in projective spaces of higher dimension, more precisely we formulate the following conjecture for $n\ge3$:}

\begin{conjecture}
Fix an integer $c\ge 2$. Then there exists an integer $d_0(c,n)$ such that for all $d\ge d_0(c,n)$ there are integers $x_i$, $1\le i\le c$, and $y_i$, $1\le i\le c-1$, such
that \begin{itemize}
\item
$x_1<y_1<x_2<y_2<\cdots <x_{c-1} <y_{c-1} < x_c,$ 
\item $\TT(n,d;x_i)'\ne \emptyset$ for all $i=1,\dots ,c$, and
\item
 $\TT(n,d;y_i)'=\emptyset$ for all $i=1,\dots ,c-1$.
\end{itemize}
\end{conjecture}

{The techniques we use in this paper are based mainly on the two notions of {\it critical schemes} for a Terracini set, and {\it numerical character} of a zero dimensional scheme in the plane. The numerical character  is a combinatorial description of the Hilbert function, which is specific for the plane and has useful consequences. In particular we prove in Theorem \ref{due00} that the numerical character of any critical scheme for a Terracini set is connected, and this fact is crucial in our main proof.
}

{
In the following section we recall some preliminary results, we define the critical schemes (Definition \ref{critic}) and we describe their properties. We also explain what is the numerical character of zero-dimensional schemes. The last Section \ref{secpro} is devoted to the proof of our main Theorem \ref{due2}.
}

{We thank the referee for pointing out a mistake in the first version of this paper and for useful comments.}
 
\section{Preliminaries}
We work over an algebraically closed field of characteristic $0$.

Given a scheme $X\subset \PP^n$, we denote by $|\Ii_X(d)|:=\PP(H^0(\PP^n,\Ii_X\otimes\Oo(d)))$ the linear system of hypersurfaces of degree $d$ containing $X$.

We recall first some preliminary facts.
\begin{remark}\label{o101}
 Let $C\subset \PP^2$ {be a curve of degree $t\le d$} and $Z\subset C$ be a zero-dimensional scheme.  Since $h^1(\Oo_{\PP^2}(d-t) )=0$, the restriction map $H^0(\Oo_{\PP^2}(d))\to H^0(\Oo_C(d))$ is surjective. 
Thus $h^1(\Ii_Z(d)) =h^1(\Ii_{Z,C}(d))$. 
Moreover we have
$$h^0(\Oo_C(d)) =\binom{d+2}{2} -\binom{d-t+2}{2} =\frac12t(2d+3-t )$$ and 
$$h^0(\Oo_C(d+1))  =\frac12t(2d+5-t ).$$
\end{remark}

\begin{remark}\label{a9}
Let $W\subset Z\subset \PP^n$ be zero-dimensional schemes. Since  $h^i(\Ii_Z(t)) =0$ for all $i\ge2$ and $h^i(\Ii_{W,Z}(d)) =0$ for all $i\ge 1$, we have:
$$h^0(\Ii_Z(d))\le h^0(\Ii _W(d))\mbox{\ \ and \ }h^1(\Ii _W(d)) \le h^1(\Ii _Z(d)),$$
and
$$h^0(\Ii_Z(d))\le h^0(\Ii _Z(d+1))\mbox{\ \ and \ }h^1(\Ii _Z(d+1)) \le h^1(\Ii _Z(d)).$$
\end{remark}

\begin{notation}\label{notationstau}
For any zero-dimensional scheme $Z\subset \PP^n$ such that $Z\ne \emptyset$, let 
$\tau(Z)$ be the maximal integer $t$ such that $h^1(\Ii_Z(t)) >0$ and
$s(Z)$ be the minimal integer $t$ such that $h^0(\Ii_Z(t))> 0$. 
In other words $s(Z)$ is the minimal degree of a hypersurface containing $Z$.
\end{notation}

Note that since $h^1(\Ii_Z(t))>0$ for any $t\le-1$, then $\tau(Z)\ge -1$ and $\tau(Z)=-1$ if and only if $\deg (Z) =1$.
\begin{proposition}
For any finite and non-empty set $S\subset\PP^n$ we have $$s(S)<s(2S)\le 2s(S).$$
\end{proposition}

\begin{proof} 
If $D$ is a hypersurface in the linear system $ |\Ii_S(s(S)|$, then $2S\subset 2D$ and hence $s(2S) \le 2s(S)$.

On the other hand, by Remark \ref{a9}
 it follows $s(S)\le s(2S)$. 
 Now we prove  that $s(S)< s(2S)$.
Assume by contradiction that $s(S)=s(2S)$. 
Then there is a hypersurface $D\in |\Ii_S(s(S))|$ such that $S\subseteq \mathrm{Sing}(D)$. Let $D$ be defined by $f=0$ for some form $f$ of degree $d$. Fix coordinates $(x_0:\cdots:x_n)$ such that  $\partial _{x_n}(f)(0:\cdots:1) \ne 0$.
Then the degree $s(S)-1$ hypersurface defined by $\partial_{x_n}(f)=0$ contains $S$, which contradicts the minimality of $s(S)$. 
\end{proof}
In Example \ref{due001}  we will see that it is possible that $s(2S) =2s(S)$ even for $S\in \TT(2,d;x)$.

\smallskip

The {\it critical schemes} are  crucial tools in our proofs.
We recall  their definition from \cite{Terracini1}:
\begin{definition}\label{critic}
Given 
a collection $S$ of $x$ points in $\PP^n$,
we say that a zero-dimensional scheme $Z$ is {\it $d$-critical for $S$} if:
\begin{itemize}
\item $Z\subseteq 2S$ and any connected component of $Z$ has degree  $\le 2$,
\item $h^1(\Ii_Z(d))>0$,
\item $h^1(\Ii_{Z'}(d))=0$ for any  $Z' \subsetneq Z$.
\end{itemize}
\end{definition}

{By the so-called curvilinear lemma (see e.g. \cite[Lemma 2.9]{Terracini1}), we have that for any $S\in\TT(n,d;x)$ there exists a $d$-critical scheme for $S$.}

Recall that, by \cite[Lemma 2.12]{Terracini1}, if $Z$ is critical for $S$, then $S\subseteq Z$. 



\begin{proposition}\label{cita1} 
Given $S\in \TT(n,d;x)'$ and any critical scheme $Z$ of $S$,
then   we have
 \begin{itemize}
\item $\tau(2S) =d$ and $\tau(Z)=d$ 
\item $s(S)\le s(Z)\le s(2S)$
\end{itemize}
\end{proposition}
\begin{proof} By {\cite[Theorem 3.1]{Terracini1}} we have $\tau(2S) =d$.   
Hence, by the curvilinear lemma, we have that $\tau(Z)=d$ for every critical scheme $Z$ of $S$.
The second statement follows from Remark \ref{a9}, since $S\subseteq Z\subseteq 2S$.
\end{proof}

The following technical result will be useful in the sequel.
\begin{lemma}\label{d5.0101} 
Let $W\subset \PP^2$ be a complete intersection of a plane curve $A$ of degree $a$ and a plane curve $B$ of degree $b$. Then $\deg (W) =ab$, 
$$h^1(\Ii _W(t)) =0 \quad \text{ for all }\quad t\ge a+b-2,$$ 
$$h^1(\Ii _W(a+b-3))=1\text{ and }h^1(\Ii _{W'}(a+b-3)) =0\text{ for all }W'\subsetneq W.$$
\end{lemma}

\begin{proof}
B\'ezout gives $\deg (W) =ab$. Since $h^2(\Oo_{\PP^2}(-3)) =1$ and $h^2(\Oo _{\PP^2}(t)) =0$ for all $t\ge -2$, the Koszul complex associated to the equations of the plane curves $A$ and $B$ gives $h^1(\Ii _W(t)) =0$ for all $t\ge a+b-2$ and  $h^1(\Ii _W(a+b-3))=1$. 

To prove that $h^1(\Ii _{W'}(a+b-3)) =0$ for all $W'\subsetneq W$ it is sufficient to prove it for all $W'$ of degree $ab-1$.
 Fix $W'\subset W$ such that $\deg (W')=ab-1$. The curves $A$ and $B$ links $W'$ to a point, $p$. The  linkage formula (\cite[Lemma at p. 199]{harris}, \cite[\S 3]{ps}) gives $h^1(\Ii _W(a+b-3)) =h^1(\Ii_p)=0$.
\end{proof}

\begin{notation}\label{lemma-numerico}
Given an integer $d$, we define the functions $$f(t):= t(d+3-t)/2$$
and $$g(t):=t(d+(5-t)/2).$$
Note that the  functions $f$ and $g$ are strictly increasing in the intervals $(0,(d+3)/2)$ and $(0,d+(5/2))$ respectively.
\end{notation}

We recall now the crucial notion of  {\it numerical character}, which is a way to encode the Hilbert function of a zero-dimensional scheme $Z$ of $\PP^2$, see \cite{e,ep,gp}.

The Hilbert function of any subvariety $Y\subset \PP(V)$ with homogeneous ideal $I(Y)$ is defined as
$$H_Y: \NN\to \NN, \quad H_Y( t)= \dim(\CC[V]_t/I(Y)_t).$$
and its first difference is
$$\Delta_Y (t) := H_Y(t) -H_Y(t-1).$$

Assume now that $Z\subset \PP^2$ is a zero-dimensional scheme of degree $z$ and
 set $s:= s(Z)$ and $\tau:= \tau(Z)$ as in Notation \ref{notationstau}. 

For all $t\in \NN$, we consider the sequence $\{h^1(\Ii_Z(t))\}_{t\ge 0}$. Since $h^1(\Ii_Z(t))=z-H_Z(t)$, we have
$$\Delta_Z (t)= h^1(\Ii_Z(t-1)) -h^1(\Ii _Z(t))\ge0.$$

The knowledge of the sequence of integers $\Delta(t)$ is equivalent to the knowledge of the sequence of integers $n_0, n_1, \cdots , n_{s-1}$, called the {\it numerical character} of $Z$, such that
\begin{itemize}
\item $n_0\ge n_1\ge \cdots \ge n_{s-1}\ge s$,
\item  for all $t\ge 0$ 
\begin{equation}\label{eqep1}
h^1(\Ii_Z(t)) = \sum _{i=0}^{s-1}((n_i-t-1)_+ -(i-t-1)_+),
\end{equation}
where we denote by $a_+:= \max \{a,0\}$ for any $a\in\ZZ$.
\end{itemize}

In particular, since $h^1(\Ii _Z)=z-1$, we have $z =\sum _{i=0}^{s-1} (n_i-i)$, i.e.
\begin{equation}
\sum _{i=0}^{s-1} n_i=z+ \binom{s}{2}.
\label{sum-n_i}
\end{equation}

Since $n_0\ge n_1\ge \cdots \ge n_{s-1}\ge s$, from \eqref{sum-n_i} we also obtain:
\begin{equation} \label{numerico}
z\ge s^2-\binom{s}{2}=\binom{s+1}{2}.
\end{equation}

Obviously $\Delta_Z (t) =t+1$ if $t<s$. For all $t\ge s$ from \eqref{eqep1} one sees that ${\Delta_Z} (t)$ is the number of integers $i$ such that $n_i\ge t+1$. 

Note also that $n_0 =\tau+2$, by \cite{e}. 
{Then from \eqref{sum-n_i} we also easily deduce the following inequality:
\begin{equation}\label{eq-g}
z\le s(\tau+2)-\binom{s}{2}=s\left(\tau-\frac{5-s}2\right).
\end{equation}}

The numerical character $n_0\ge \cdots \ge n_{s-1}$ is said to be {\it connected} if $n_i\le n_{i+1}+1$. 
By \eqref{eqep1}, the numerical character of $Z$ is connected if and only if the {sequence $\{\Delta_Z(t)\}$ is strictly decreasing in the range $\sigma_0\le t\le \tau$, where $\sigma_0$ is the minimal $\sigma$ such that $\Delta_Z(\sigma)>\Delta_Z(\sigma+1)$.

\begin{remark}\label{sotto-curva}
Let $Z$ be a zero-dimensional scheme in the plane and let $n_0\ge \cdots \ge n_{s-1}$ be its numerical character.
Then, by \cite[Proposition at p. 112]{ep} (see also \cite[Prop. 4]{e}), if $n_{t-1}> n_t+1$, then 
 there is a degree $t$ plane curve $C$ such that the scheme 
 $Z':= Z\cap C$ has numerical character $n_0\ge \cdots \ge n_{t-1}$. 
 Moreover
the scheme $\Res_C(Z)$ has numerical character $m_0,\dots ,m_{s-t-1}$ with $m_i =n_{t+i} -t$.
\end{remark}

There is a characterization of the connectedness of the numerical character of a zero-dimensional scheme $Z\subset \PP^2$ in term of the minimal free resolution of $Z$ (\cite[Theorem 7]{e}).

The following lemma will be crucial in the following proofs.

\begin{lemma}\label{z03}
Given  a zero-dimensional scheme $W\subset \PP^2$ of degree $w>0$, set $d=\tau(W)$ 
and $s= s(W)$.
 Assume that 
 \begin{itemize}
 \item[(a)] $s\le (d+3)/2$, 
 \item[(b)]
 the numerical character $n_0,\dots ,n_{s-1}$ is connected,
 \end{itemize}
 and that there is a positive integer $a$ such that 
 \begin{itemize}\item[(c)] $a^2\le w$ and 
\item[(d)] $d\ge a-3 +\frac{w}{a}$.
\end{itemize}
Then 
\begin{itemize}
\item[(i)] either $a=s$ and $w = s(d+3-s)$,
\item[(ii)] or $s<a$ and there is an integer $0<m<a$ and a degree $m$ curve $C$ such that $\tau(W\cap C)=d$.
\end{itemize}
\end{lemma}


\begin{proof} 
From the assumptions $w\le a(d+3-a)$ and $a^2\le w$, it follows that $a\le (d+3)/2$. 

We prove first that $s\le a$.

Indeed since $n_0,\dots ,n_{s-1}$ is connected and $n_0=\tau(W)+2=d+2$, then $n_i\ge d+2-i$ for all $i$. 
By \eqref{sum-n_i} we have 
\begin{equation}\label{bella}
w\ge s(d+2) -s(s-1) =s(d+3-s).\end{equation}
Recall, from Notation \ref{lemma-numerico}, that the function $f(t)= t(d+3-t)$ is strictly increasing in the interval $(0,(d+3)/2)$.
Since $w\le a(d+3-a)$ and recalling that $0<s\le (d+3)/2$ and $0<a\le (d+3)/2$, we get $s\le a$.

\quad(i)
Now assume that $s=a$. In this case by \eqref{bella} we get $w =s(d+3-s)$.

\quad(ii)
Now assume $s<a$. By \cite[Cor 2]{ep} there is an integer $0<m<a$ and a degree $m$ curve $C$ such  that $\tau(W\cap C)=\tau(W)=d$.
\end{proof}

Notice that in case (i) of the previous lemma the scheme $W$ can be a complete intersection of a curve of degree $s$ and a curve of degree $w/s$. Anyway, as pointed out by the referee, there are also example of schemes having the Hilbert function of a complete intersection, but which are not complete intersection.

\smallskip

The following theorem is an important property of the numerical character of a critical scheme.

\begin{theorem}\label{due00} 
Fix $S\in \TT(2,d;x)'$ and take any critical scheme $Z$ for $S$. 
Let $s=s(Z)$ 
and let $n_0\ge \cdots \ge n_{s-1}$ be the numerical character of $Z$.
Then
$n_0=d+2>n_1$ and the numerical character $n_0\ge \cdots \ge n_{s-1}$ is connected.
\end{theorem}

\begin{proof} We have $\deg (Z) = \sum _{i=0}^{s-1} (n_i-i)$ and $n_i-i>0$ for all $i$. 
{By Lemma \ref{cita1}, we have} $\tau(Z)=d$ and hence we have $n_0=d+2$. 
Since $Z$ is critical, by {\cite[Lemma 2.11]{Terracini1},} we know that $h^1(\Ii_Z(d)) =1$,  and hence $n_1<n_0$.
 
 Assume by contradiction that the numerical character is not connected and call $t$ the first integer $\le s-1$ such that $n_{t-1}\ge n_t+2$.
{By Remark \ref{sotto-curva},} there is a degree $t$ plane curve $C$ such that the scheme $Z':= Z\cap C$ has numerical character $n_0\ge \cdots \ge n_{t-1}$. 
Since $n_0$ is the first element of the numerical character of $Z'$ and $d=n_0-2$, then $h^1(\Ii _{Z'}(d)) >0$. Since $Z$ is critical, $Z =Z'$ and hence $t=s$, which is a contradiction. 
\end{proof}

\begin{lemma}\label{NEW}
Given $S\in \TT(2,d;x)'$ and a critical scheme $Z$ of $S$, let  $s:= s(Z)$ and $n_0,\dots ,n_{s-1}$ be the numerical character of $Z$. Take a curve $C\in |\Ii_Z(s)|$ and
 a zero-dimensional scheme $W$ such that $Z \subseteq W\subset C$
and each connected component of $W$ has degree $2$.

Then $\tau(W)=d$, $s(W)=s$ and the numerical character of $W$ is connected.
\end{lemma}
\begin{proof}
Let $c=s(W)$ and $m_0,\ldots, m_{c-1}$ be the numerical character of $W$.

By construction $c\le s$. 
Recall by Proposition \ref{cita1} that $\tau(2S)=\tau(Z)=d$. Since $2S\supset W\supseteq Z$,  then  by Remark \ref{a9} we have $\tau(W)=d$, and this implies that $m_0=n_0$.

Now we assume by contradiction that either $c<s$, or $m_0,\dots ,m_{c-1}$ is not connected.
Then there is an integer $r<s$ and a degree $r$ curve $D$ such that $Z\nsubseteq D$ and $\tau(W\cap D)=d$, contradicting the minimality of $S$.
\end{proof}

The following example shows that, given $S\in \TT(2,d;x)'$, the numerical character of $2S$ is not alway connected. 

\begin{example}\label{due001} 
Fix an integer {$d\ge 8$. Let $D\subset \PP^2$ be a smooth conic, and let $S$ be a collection of  $d+1$ points on the conic $D$. 
We have  $\tau (S) =\lceil {d/2}\rceil -1$.
Since $D$ is the minimal degree curve containing $S$, {we have $s(S)=2$} and the numerical character of $S$ is $m_0,m_1$ with $m_0 =\lceil d/2\rceil {+1}$ and hence $m_1 =\deg (S)+1-m_0 = d-\lceil d/2\rceil{+1}$.

Hence we have:
\begin{center}
\begin{tabular}{|c|c|c|c|c|c|c| } 
\hline
 t&0 & 1 & 2 & $\cdots$&$\lceil d/2\rceil-1$  & $\lceil d/2\rceil$\\ 
\hline
$h^1(\Ii_S(t)) $  &$d$ & $d-2$ & $d-4$ & $\cdots$&$d+2-2\lceil d/2\rceil$&0\\  
\hline
$\Delta_S(t)$  &1 & 2 & 2 & $\cdots$&2& $d+2-2\lceil d/2\rceil$\\ 
 \hline
\end{tabular}
\end{center}

Consider now the residual exact sequence
of $D$:
\begin{equation}\label{eqbb01}
0 \to \Ii_S(t-2) \to \Ii_{2S}(t) \to \Ii _{(2S,D),D}(t)\to 0.
\end{equation}
Since $h^1(\Ii_S(d-2)) =0$ {(since $d\ge 4$)}
 and $\deg (2S\cap D) =2d+2$, we get 
 $$h^1(\Ii _{2S}(d)) ={h^1(\Ii _{(2S,D),D}(d))=h^0(\Oo_{\PP^1})}=
 1$$ and $S\in \TT(2,d;d+1)'$. 
Hence by Theorem {\cite[Theorem 3.1]{Terracini1}}, we have $h^1(\Ii _{2S}(d+1))=0$ and $\tau(2S)=d$.

Now we compute the numerical character of $2S$.
Since $d\ge 3$,
then the quartic
$2D$ is the minimal degree curve containing $2S$. Hence $s(2S)=4$ and let $n_0,n_1,n_2,n_3$ be the numerical character of $2S$. We have $n_0=\tau+2=d+2$.}
From \eqref{eqbb01} with $t=d-1$ we get $h^1(\Ii_{2S}(d-1))=3$, hence we have $n_1=d+1$.
Again from \eqref{eqbb01} with $t=d-2$, we get $h^1(\Ii_{2S}(d-2))=5$, hence we have $n_2\le d-1$.

Thus $n_0,n_1,n_2,n_3$ is not connected
and Remark \ref{sotto-curva} gives 
$$n_0=d+2, n_1=d+1, n_2 =m_0+2= \lceil d/2\rceil {+3}, n_3=d-\lceil d/2\rceil+3.$$
\end{example}


\section{The proof of Theorem \ref{due2}}
\label{secpro}
In this section we  prove our main result, Theorem \ref{due2}. 

The  following proposition gives the non-emptyness of the Terracini loci for suitable number of points.
\begin{proposition}\label{o1o1} 
Fix an integer $t\ge 2$ and an integer $d$ 
such that \begin{itemize}
\item
 $d+3-t$ is even and 
 \item $d\ge 3t-1$.
 \end{itemize} Set $x:= t(d+3-t)/2$.   
Then there exists $S\in \TT(2,d;x)'$ which admits a critical scheme of degree $2x$.
\end{proposition}

\begin{proof}  Let $S\subset \PP^2$ be a set of  $x$ non-collinear points which is the complete intersection
of a smooth curve $C$ of degree  $t$ and a curve of degree $(d+3-t)/2$. Set $Z:= C\cap 2S$. 
Since $Z$ is the complete intersection of $C$ and a curve of degree $d+3-t$,
by Lemma \ref{d5.0101} we have
$h^1(\Ii_{Z}(d)) =1$ and $h^1(\Ii_{Z'}(d)) =0$ for all $Z'\subsetneq Z$.  
 Since $C$ is smooth, for any $A\subseteq S$ the residual exact sequence of $C$ is:
\begin{equation}\label{eqeso1}
0\to \Ii_A(d-t)\to \Ii_{2A}(d)\to \Ii_{{2A\cap C,C}}(d)\to 0\end{equation}
Thus to prove that $h^1(\Ii_{2S}(d)) =1$ 
and that $h^1(\Ii_{2A}(d)) =0$ for all $A\subsetneq S$ (and hence to prove that $S$ is minimal), it is sufficient to prove that
$h^1(\Ii_{S}(d-t)) =0$. This is true by Lemma \ref{d5.0101}, because $S$ is the complete intersection of a curve of degree $t$ and a curve of degree $(d+3-t)/2$ and $d-t\ge t+(d+3-t)/2 -2$ by assumption. 
\end{proof}

We are finally ready to prove our main
result.
\begin{proof}[Proof of Theorem \ref{due2}:]
Given $c\ge2$, we fix an integer $d$ such that
\begin{equation}
d\ge 14c+2.
\label{condizione}
\end{equation}

\quad{\bf Step 1.} 
First we define the numbers $x_i$, for $i=1,\dots ,c$, and we prove that
$\TT(2,d;x_i)'\ne \emptyset$ .

Consider the function $f(t)= t(d+3-t)/2$, defined in Notation \ref{lemma-numerico} and recall that it is strictly increasing in the interval $(0,(d+3)/2)$.

For any $i=1,\ldots, {c}$,  we define:
$$x_i=\left\{\begin{array}{ll}
f(2i)=i(d+3-2i)&\textrm{ if $d$ is odd},\\ 
f(2i+1)=(2i+1)(d+2-2i)/2 &
\textrm{ if  $d$ is even}\end{array}\right.$$
Note that the numbers $x_i$ are the integer values of $f(t)$ in the interval $2\le t\le 2c+1$.
Since  {$2c+1\le(d+3)/2$} by assumption \eqref{condizione}, then we have $$ x_1< \ldots< x_c. $$


Setting $t=2i$ if $d$ is odd and $t=2i+1$ if $d$ is even, it is easy to check that the assumptions of 
 Proposition \ref{o1o1} are satisfied. Indeed the second condition is verified since {$d\ge6c+2$} by \eqref{condizione}.
 
Then it follows by Proposition \ref{o1o1} that
$\TT(2,d;x_i)'\ne \emptyset$. 

\smallskip

\quad{\bf Step 2.} 
Now  we define the numbers $y_i$, for $i=1,\dots ,c-1$, and we prove to that
$\TT(2,d;y_i)'= \emptyset$ by contradiction.

For any $i=1,\ldots,c-1$, let us define $$y_i:= x_{i+1}-1.$$ 
We clearly have  $x_1<y_1<x_2<\cdots  <y_{c-1} < x_c.$

Assume by contradiction the existence of $S\in \TT(2,d;y_i)'$. 

We have $\deg (2S) =3y_i$ and $\tau(2S) =d$, {by Lemma \ref{cita1}}.

Let $Z$ be a critical scheme for $S$. 
Now let  $s= s(Z)$ and take a curve $C\in |\Ii_Z(s)|$.
Let
$W$ be a zero-dimensional scheme such that $Z \subseteq W\subset C$,
$W_{\red}=S$ and each connected component of $W$ has degree $2$. 
Then by Lemma \ref{NEW} we know that  $\tau(W)=d$, $s(W)=s(Z)=s$.
Clearly $w:=\deg(W)=2y_i$.

{\bf \quad Case I:}
Let us assume first that $d$ is odd.
Hence
$x_i=
f(2i)=i(d+3-2i)$ and 
$y_i=x_{i+1}-1=
(i+1)(d+1-2i)-1$.

Set $a=2i+2$.
Since we want to apply Lemma \ref{z03} to $W$ , we  check the  following four hypotheses:

\quad(a) 
$s\le (d+3)/2$. 

Indeed if by contradiction we assume $s>(d+3)/2$, then by \eqref{numerico} we would have
$w\ge(d+5)(d+3)/8$.
On the other hand, since $w=2y_i\le 2y_{c-1}$, we have
$$(d+5)(d+3)/8\le 2c(d+3-2c)-2$$
which is false as soon as 
$$d>
8c-4+\sqrt{32c^2-16c-15},$$
which is true by \eqref{condizione}.

\quad(b) 
The numerical character of $W$ is connected by Theorem \ref{due00} and Lemma \ref{NEW}.

\quad(c) 
$w\ge a^2$. 

Indeed $y_i-a^2=(i+1)(d+1-2i)-1-(2i+2)^2=(i+1)(d-6i-3)-1\ge0$ 
because, by \eqref{condizione},
$d>6c-3> 6i+3$. 
Then we have
$w=2y_i\ge y_i\ge a^2$.

%

\quad(d) 
 $\tau(W) >  a-3+w/a$.

Indeed, 
$d >  a-3+w/a$ is equivalent to 
$2f(a)>w$. Since $f(a)=x_{i+1}$ and $w= 2y_i=2(x_{i+1}-1)$, then
we conclude that $\tau(W) >  a-3+w/a$.

%
%
%

\smallskip
 
Then, we  can apply Lemma \ref{z03} and we have the following two possibilities:
  
{\quad(i)}
  either $a=s$ and $w=s(d+3-s)$. In this case  we get to a contradiction, because we know that $w<a(d+3-a)$ by condition (d);

{\quad(ii)
or we have $s<a$ and there is an integer $0<m<a$ and a degree $m$ curve $D$ such that
$\tau(W\cap D)=d$. }

Recall that $Z$ is a critical scheme for the minimal Terracini set $S$ and $S\subseteq Z\subseteq W\subset 2S$.

We now prove that $S\subset D$. Indeed if $S\cap D=S'\neq S$, then we would have
$h^1(\Ii_{2S'}(d))\ge h^1(\Ii_{W\cap D}(d))>0$, which contradicts the minimality of $S$.
 

Let $Z'$ be a   subscheme of $W\cap D$  such that $h^1(\Ii_{Z'}(d))>0$ and $h^1(\Ii_X(d))=0$ for any $X\subset Z'$.  Then $Z'$ is critical for $S$.

Let $W'$ be a zero-dimensional  scheme
such that
that $Z'\subset W\cap D \subseteq W'\subset D$,
and each connected component of $W$ has degree $2$.



Clearly, we have
$\deg(W')=2y_i$ and $s':=s(W')\le m$. 
Moreover by Lemma \ref{NEW} we have $\tau(W')=d$.
Then by \eqref{eq-g} we have
$$2y_i\le 
s'(d+(5-s')/2)=g(s'),$$ where the function $g$ is defined in Notation \ref{lemma-numerico} and is strictly decreasing in the interval $(0,d+(5/2))$.
Since $$0<s'\le m<2i+2=a\le d+5/2,$$
where the last condition holds by condition \eqref{condizione} since $i\le c$, it follows
that
$$2y_i\le 
g(2i+1).$$
On the other hand it is easy to check that $2y_i> g(2i+1)$
and this give a contradiction.

\medskip

{\bf \quad Case II:}
Now let us assume that $d$ is even.
Hence
$x_i=
f(2i+1)=(2i+1)(d+2-2i)/2$ and 
$y_i=x_{i+1}-1=
(2i+3)(d-2i)/2-1.$

In this case we set $a= 2i+3$ and we check again the hypotheses of Lemma \ref{z03} for the scheme $W$ as follows:

\quad(a) $s\le (d+3)/2$. 

Indeed if by contradiction we assume $s>(d+3)/2$, then, arguing as in the previous case, we would obtain
$$(d+5)(d+3)/8\le
(2c+1)(d-2c+2)-2
$$
which is false as soon as 
$$d>8c+\sqrt{32c^2+16c+17},$$
which is true by \eqref{condizione}.

\quad(b)  The numerical character of $W$ is connected by Theorem \ref{due00} and Lemma \ref{NEW}.

\quad(c) $w\ge a^2$.

Indeed $y_i-a^2=(2i+3)(d-2i)/2-1-(2i+3)^2=(2i+3)(d-6i-6)/2-1
\ge0$
because, by assumption, $d> 6c+6> 6i+6$. 
Then we have
$w\ge y_i\ge a^2$.

\quad(d) $\tau(W) >  a-3+w/a$.

Indeed, 
$d >  a-3+w/a$ is equivalent to 
$2f(a)>z$. Since $f(a)=x_{i+1}$ and $w= 2y_i=2(x_{i+1}-1)$
we conclude.


Then we can apply Lemma \ref{z03} and following the same argument of Case I, we get again to  a contradiction. This conclude the proof of the theorem.
\end{proof}

\end{document}